\documentclass[12pt]{amsart}
\usepackage{amsthm,amsfonts,amsmath,amssymb,amsaddr,geometry,eufrak,bigints,eucal}
\usepackage{graphicx}
 \usepackage{enumitem}
 \usepackage{hyperref}
 \usepackage{etoolbox}
 \usepackage{doi}
 \makeatletter
\patchcmd{\section}{\scshape}{\relax}{}{}
\patchcmd{\section}{\uppercase}{\relax}{}{}
\makeatother
\input xypic
\usepackage{url}
\newtheorem{theorem}{Theorem}[section]

\newtheorem{corollary}{Corollary}[section]

\newtheorem{definition}{Definition}[section]
\newtheorem{example}{Example}[section]

\newtheorem{remark}{Remark}[section]

\numberwithin{equation}{section}

\begin{document}

\title[]{Another Perspective on Chatterjea Contraction}

\author{Shallu Sharma$^1$, Irfan Ahmed$^2$ and Sahil Billawria$^{3,*}$}
\maketitle

\begin{center}
\footnotesize $^1$Department of Mathematics, University of Jammu, Jammu, India.\\
Email: shallujamwal09@gmail.com\\
$^2$Department of Mathematics, University of Jammu, Jammu, India.\\
Email: irfah.ahmed@jammuuniversity.ac.in\\
$^3$Department of Mathematics, University of Jammu, Jammu, India.\\
Email: sahilbillawria2@gmail.com \\
\end{center}
\begin{abstract}
Inspired by the well-known result stating that if any iterate of a mapping is a Banach contraction on a complete metric space, then the mapping itself possesses a unique fixed point, we investigate that claim for a Chatterjea contraction but by retaining the left-hand side of the inequality as per the mapping itself. With an additional assumption of $k$- continuity, the existence and uniqueness of a fixed point is obtained for a new class of contractions, $m$-Chatterjea contraction, on a complete metric space. Several examples are given in order to substantiate many theoretical claims such as discontinuity at the unique limit point of the iterative sequence, as well as examples demonstrating that this new class strictly contains the class of Chatterjea mappings.
\end{abstract}
\section{Introduction and Preliminaries}
The Banach fixed point theorem \cite{Banach} stands as one of the most fundamental result in metric fixed point theory, with numerous generalizations and corollaries. Among its well-known consequences is the fact that if some iterate of a mapping is a Banach contraction on a complete metric space, then the mapping itself possesses a unique fixed point. This paper investigates an analogous claim for the Chatterjea contractive condition which was introduced by S. K. Chatterjea \cite{Chatterjea} in 1972 and provides conditions under which a mapping possesses a unique fixed point.
\begin{theorem}
 Let $(X,\rho)$ be a complete metric space and $\psi:X \to X$ is a self-mapping with $\alpha \in \left[0, \frac{1}{2}\right)$ such that for all $x, y\in X$, the following holds:
\begin{equation}
\rho(\psi x, \psi y) \le \alpha \left( \rho(x, \psi y) + \rho(y, \psi x) \right).\tag{1.1}\label{1.1}
\end{equation}
Then the mapping $\psi$ has a unique fixed point in $X$.
\end{theorem}
The Chatterjea fixed point theorem occupies a distinguished position in the literature due to the lack of continuity aassumption and through the carachterization of completeness of an underlying metric space. Moreover, Kannan \cite{Kannan} and Chatterjea \cite{Chatterjea} independently provided fixed point theorems that characterize completeness of metric spaces, a property not shared by Banach's theorem (as demonstrated by Connell \cite{Connell}) and hence refuted the claim that a metric space $(X,\rho)$ is complete if and only if any Banach contraction on $X$ possesses a fixed point. But in 1975, this claim was proven valid for the class of Chatterjea contractions by Subrahmanyam \cite{Sub}. Moreover, the uniqueness of the fixed point  is the only thing that unites Chatterjea mappings, Kannan mappings, and the Banach Contraction Principle, otherwise they are independent of each other.

Within the field of fixed point theory, the class of Chatterjea mappings has been extensively studied and modified in numerous ways that can be distinguished from one another. The contractive property of the mapping is loosened in the first instance (See \cite{CA, CB, CC, CD, CH}). The second instance involves relaxing the topology as examined in \cite{CE}. And the third instance concerns theorems developed for multivalued mappings of the Chatterjea type (See \cite{CF, CG}). Finally, the fourth case examines nontrivial extensions formulated in a weaker or more generalized metric framework, thereby providing a broad and systematic investigation of the various directions in which this theory can be extended (See\cite{CI, CJ, CK}).

In this paper, we present a modified perspective. While it is obvious that if the $n$-th iterate of a mapping satisfies the Chatterjea condition then, by the same approach as for Banach contractions, the mapping has a unique fixed point on a complete metric space, we instead retain the right-hand side of the inequality (\ref{1.1}) in its original form and consider the $m$-th iterate on the left-hand side. This leads to the notion of an $m$-Chatterjea mappings fulfilling the condition
\[
\rho(\psi^m x, \psi^m y) \le \alpha \left( \rho(x, \psi y) + \rho(y, \psi x) \right)
\] on $X$ for some $\alpha\in\left[0,\frac{1}{2}\right)$ and $m\in \mathbb{N}$. For $m=1$, this reduces to Chatterjea contraction, so this problwm will not be in the scope of this article. we will use different proof techniques  for $m=2$ and $m\geq 3$, but some continuity assumptions are enforced in order to obtain existence of the fixed point. This differentiated it from the Chatterjea contraction, but it also opens an interesting problem, does there always exists some $k\in\mathbb{N}$ such that a Chatterjea contraction is a $k$-continuous mapping. The continuity presumption is sufficient, but not necessary as can be seen through examples. Important property of this class of mappings for any natural m is that the iterative sequence con- verges and that the limit is uniquely determined for arbitrary initial point in a complete metric space. Additionally, there are examples of mappings satisfying above mentioned contractive condition and not possessing a fixed point in a complete metric space due to being discontinuous.
We begin by recalling the concept of $ k$-continuity, which plays a crucial role in our existence results.

\begin{definition}
If $(X, \rho)$ is a metric space and $\psi: X \to X$ a mapping, then $\psi$ is $k$-continuous for some $k \in \mathbb{N}$ if $\psi^k$ is continuous.
\end{definition}

Recall that continuity of a mapping on a metric space is equivalent to sequential continuity. Clearly, continuity implies $k$-continuity for any $k \in \mathbb{N}$, but the converse does not hold for $k \ge 2$.
\begin{example}
If $X = \mathbb{R}$ be equipped with the Euclidean metric $\rho$ and define $\psi: X \to X$ by 
\[
\psi x = 
\begin{cases} 
1, & x \in \mathbb{Q},\\ 
0, & x \notin \mathbb{Q}.
\end{cases}
\]
Then $\psi$ is discontinuous on $X$, but $\psi$ is a $2$-continuous mapping ($\because \psi^2$ is constant and hence continuous).
\end{example}
\section{Main Results}
We distinguish the analysis of iterates of a mapping for $m=2$ and for $m>2$ due to differences in the underlying proof strategies, which vary slightly from those used in the general case.
\begin{theorem}\label{2.1}
If $(X, \rho)$ is a complete metric space and $\psi : X \to X$ is a $k$-continuous mapping for some $k \in \mathbb{N}$ such that for all $x, y \in X$,
\begin{equation}
\rho(\psi^2 x, \psi^2 y) \leq \alpha \left( \rho(x, \psi y) + \rho(y, \psi x) \right) \tag{2.1}\label{2.1}
\end{equation}
holds for some $\alpha \in \left(0, \frac{1}{2}\right]$. Then $\psi$ has a unique fixed point in $X$ and for any arbitrary initial point $x \in X$ the iterative sequence $(\psi^n x)$ converges to this fixed point.
\end{theorem}
\begin{proof}
Let $x_0 \in X$ be arbitrary and define $x_n = \psi^n x_0$ for any $n \in \mathbb{N}$.\\
If $\alpha = 0$, then $\psi^2$ is a constant mapping, so there exists a unique $x^* \in X$ such that $\psi^2 x = x^*$ for all $x \in X$. In particular, $\psi^2 x^* = x^*$. If $\psi x^* = y$, then $\psi^2 y = \psi(\psi^2 x^*) = \psi x^* = y$ which asserts that $y = x^*$. For any $x \in X$, $\psi^n x = x^*$ for all $n \geq 2$, so the sequence $(\psi^n x)$ converges to $x^*$.\\
 If $\alpha \neq 0$, denote with $\rho_n = \rho(x_n, x_{n+1})$ a distance. First, we establish a recurrence relation for $\rho_n$.\\

For $n \geq 2$, take $x = \psi^{n-2} x_0$ and $y = \psi^{n-1} x_0$ in (\ref{2.1}). Then
\begin{align*}
\rho(\psi^n x_0, \psi^{n+1} x_0) &= \rho(\psi^2(\psi^{n-2}x_0), \psi^2(\psi^{n-1}x_0)) \\
&\le \alpha (\rho(\psi^{n-2}x_0, \psi(\psi^{n-1}x_0)) + \rho(\psi^{n-1}x_0, \psi(\psi^{n-2}x_0))) \\
&= \alpha (\rho(x_{n-2}, x_n) + \rho(x_{n-1}, x_{n-1})) \\
&= \alpha \rho(x_{n-2}, x_n) \\
&\le \alpha (\rho(x_{n-2}, x_{n-1}) + \rho(x_{n-1}, x_n)) \\
&= \alpha (\rho_{n-2} + \rho_{n-1}).
\end{align*}
Thus,
\begin{equation}
 \rho_n \le \alpha (\rho_{n-2} + \rho_{n-1}). \tag{2.2}\label{2.2}
\end{equation}
Consider the recurrence relation $a_n = \alpha (a_{n-1} + a_{n-2})$ for $n \ge 2$ with initial conditions $a_0 = \rho_0$ and $a_1 = \rho_1$. Since all arguments are non-negative, $\rho_n \le a_n$ for all $n$. The characteristic equation is $r^2 - \alpha r - \alpha = 0$ with roots
$$r_1 = \frac{\alpha + \sqrt{\alpha^2 + 4\alpha}}{2}, \quad r_2 = \frac{\alpha - \sqrt{\alpha^2 + 4\alpha}}{2}$$
Note that $0 < r_1 < 1$ for $\alpha \in \left(0, \frac{1}{2}\right)$. Since $r_1 < 1$ is equivalent to $\sqrt{\alpha^2 + 4\alpha} < 2 - \alpha$, which holds because $\alpha < \frac{1}{2}$. Also $|r_2| < r_1$. Hence
$$a_n = P r_1^n + Q r_2^n,$$
where 
\begin{align*}
P &= - \frac{-2\rho_1 + \alpha\rho_0 - \sqrt{\alpha(4+\alpha)}\rho_0}{2\sqrt{\alpha(4+\alpha)}} \\
Q &= \rho_0 + \frac{-2\rho_1 + \alpha\rho_0 - \sqrt{\alpha(4+\alpha)}\rho_0}{2\sqrt{\alpha(4+\alpha)}}.
\end{align*}
and the series $\sum_{n=0}^{\infty} a_n$ converges. Consequently, $\sum_{n=0}^{\infty} \rho_n$ converges.
Hence $(x_n)$ is a Cauchy sequence. Since $X$ is complete, there exists $x^* \in X$ such that $\lim \limits_{n \to \infty}  x_n = x^*$.\\
Let $y \in X$ be arbitrary and define $y_n = \psi^n y_0$. 
Then by using (3.1) and triangle inequality, we have
\begin{equation*}
    \rho(\psi^n x_0, \psi^n y_0) \leq \rho(\psi^n x_0, \psi^n x_0) + \rho(\psi^n y_0, \psi^n y_0), 
\end{equation*}
which gives $\lim \limits_{n \to \infty}  \psi^n y_0 = x^*$ for all $y_0 \in X$.\\
In particular,
\begin{equation}
   \lim \limits_{n \to \infty}  \psi^n x^* = x^*. \tag{2.3}\label{2.3}
\end{equation}
Since $\psi$ is $k$-continuous mapping. Therefore,
\begin{equation*}
    \psi^k x^* = \psi^k \left( \lim_{n \to \infty} \psi^n x^* \right) = \lim \limits_{n \to \infty}  \psi^{k+n} x^* = x^*.
\end{equation*}
If $k=1$, then $\psi x^* = x^*$ and we are done. If $k \geq 2$, we have to show that $x^*$ is also a fixed point of $\psi$. For any $n \in \mathbb{N}$, it follows
\begin{align*}
    \rho(x^*, \psi x^*) &= \rho(\psi^{kn} x^*, \psi^{kn+1} x^*) \\
    &\leq \rho(\psi^{kn-2} x^*, \psi^{kn-1} x^*) + \rho(\psi^{kn-1} x^*, \psi^{kn} x^*)\to 0~\text{as}~ n \to \infty.
\end{align*}
Hence $\psi x^* = x^*$.

Further, suppose $y \in X$ such that $\psi y = y$, then
\begin{align*}
    \rho(x^*, y) &= \rho(\psi^2 x^*, \psi^2 y) \\
    &\leq \alpha \left( \rho(x^*, \psi y) + \rho(y, \psi x^*) \right) \\
    &= \alpha \left( \rho(x^*, y) + \rho(y, x^*) \right) \\
    &= 2\alpha \rho(x^*, y).
\end{align*}
Since $2\alpha < 1$, we have $\rho(x,y) = 0$.
\end{proof}
                                     
\begin{remark}
The mapping $\psi$ satisfying the condition (\ref{2.1}) has a unique fixed point if and only if $\psi^2$ has a unique fixed point. Clearly, the set of fixed points of a mapping $\psi$ is a subset
of the set of fixed point of a mapping $\psi^2$ due to
\begin{equation*}
            \psi^2 x^* = \psi (\psi x^*) = \psi x^* = x^*.
\end{equation*}
Conversely, if $z$ satisfies $\psi^2 z = z$, then
            $$\psi^n z = \begin{cases} z, & \text{if } n \text{ is even} \\ \psi z, & \text{if } n \text{ is odd} \end{cases}$$

for any $n \in \mathbb{N}$. Since $\psi^n z$ is 
convergent, we must have $\psi z = z$.
\end{remark}
\begin{example}\label{A}
Let $X = [0,1]$ equipped with the Euclidean metric and defined a mapping $\psi : X \to X$ by
$$\psi x = \frac{x}{2}, \quad x \in [0,1].$$
 If we take $x=0, y=1$.Then
\[ \rho(\psi x, \psi y) = \frac{1}{2} = \rho(x, \psi y) + \rho(y, \psi x). \]
So $\psi$ is not a Chatterjea contraction on $X$. Also it is not a Kannan contraction.\\
 Now, for any $x, y \in [0,1]$, we have
\[ \rho(\psi^2 x, \psi^2 y) = \frac{|x-y|}{4}. \]
Also,
\[ \rho(x, \psi y) + \rho(y, \psi x) = \left| x - \frac{y}{2} \right| + \left| y - \frac{x}{2} \right|. \]
Thus,
\[ \frac{|x-y|}{4} \le \frac{1}{6} \left( \left| x - \frac{y}{2} \right| + \left| y - \frac{x}{2} \right| \right).\]
Thus the inequality (\ref{2.1}) holds with $\alpha = \frac{1}{6} < \frac{1}{2}$.\\
 The same mapping fails to be $2$-Kannan contraction (contrary to the claim Example $2.3$, \cite{Cvetkovic}). Indeed, for $x=0,  y=1$
\[ \rho(\psi^2 x, \psi^2 y)=\frac{1}{4}~\text{and}~\rho(x, \psi x)+\rho(y, \psi y)=\frac{1}{2}.\]
So the $2$-Kannan inequality would require
\[ \frac{1}{4} \le \alpha \frac{1}{2} \implies \alpha \ge \frac{1}{2}. \]
\end{example}
\begin{example}\label{B}
Let $X = [0, 1]$ equipped with the Euclidean metric and define $\psi : X \to X$ by
\[ \psi x = \frac{1-x}{2}, \quad x \in [0, 1]. \]
Then $\psi$ is a Kannan contraction with constant $\alpha = \frac{1}{3}$. By Remark [2.4. ], $\psi$ is a $2$-Kannan contraction.\\

 we now show that $\psi$ is not a $2$-Chatterjea contraction. \\
If $x = 0, \ y = \frac{1}{2}$. Then
\[ \rho(\psi^2 x, \psi^2 y) =  \frac{1}{8} \]
and
\[ \rho(x, \psi y) = \frac{1}{4}, \ \rho(y, \psi x) = 0. \]
Thus, $\frac{1}{8} \le \alpha \frac{1}{4}$ holds only if $\alpha \ge \frac{1}{2}$. \\
Hence $\psi$ is not a $2$-Chatterjea contraction.
\end{example}

\begin{remark}
From Example \ref{A} and \ref{B}, we can conclude that the classes of $2$-Kannan contraction and $2$-Chatterjea contractions are independent.
\end{remark}

The following corollary follows directly from the proof of Theorem \ref{2.1}, as the continuity presumption was not used to establish convergence of the iterative sequence.

\begin{corollary}
Let $(X, \rho)$ be a complete metric space and let $\psi : X \to X$ satisfy (\ref{2.1}) for some $\alpha \in \left[0, \frac{1}{2}\right)$. Then for any arbitrary initial 
point $x \in X$, the iterative sequence $(\psi^n x)$ converges in $(X, \rho)$ to the same point $x^* \in X$ and $Fix(\psi) \subseteq \{x^*\}$.
\end{corollary}

\begin{corollary}
A $2$-Chatterjea contraction on a complete metric space has at most one fixed point.
\end{corollary}
\section{$m$-Chatterjea Mappings}
The earlier notion of $2$-Chatterjea mappings exhibits several noteworthy properties and admits a natural generalization to any natural number $m$ via the concept of $m$-Chatterjea mappings.

\begin{definition}
If $(X, \rho)$ is a metric space. Then a mapping $\psi: X \to X$ is called $m$-Chatterjea contraction if there exists $\alpha \in \left[0, \frac{1}{2}\right)$ and $m \in \mathbb{N}$ with $m \geq 3$ such that for all $x, y \in X$, the following inequality holds:
\begin{equation}
\rho(\psi^m x, \psi^m y) \leq \alpha (\rho(x, \psi^m y) + \rho(y, \psi^m x)). \tag{3.1}\label{3.1}
\end{equation}
\end{definition}
\begin{remark}
For $m=1$, this reduces to the classical Chatterjea contraction. As was seen for the case $m = 2$ in Example \ref{A}, not any m-Chatterjea contraction is a Chatterjea contraction but the reverse statement holds.
\end{remark}

\begin{theorem}\label{D}
If $(X, d)$ is a metric space and let $\psi: X \to X$ be a $K$-continuous mapping for some $K \in \mathbb{R}^+$. Suppose that for some $m \in \mathbb{N}$ with $m \geq 3$ and $\alpha \in \left[0, \frac{1}{2}\right)$, the inequality (\ref{3.1})holds for all $x, y \in X$. Then $\psi$ has a unique fixed point in $X$ and for arbitrary initial point $x_0 \in X$, the iterative sequence $(\psi^n x_0)$ converges to this fixed point.
\end{theorem}

\begin{proof}
Let $x_0 \in X$ be arbitrary and define a sequence $x_n = \psi^n x_0$ for $n \in \mathbb{N}$. Denote $\rho_n = d(x_n, x_{n+1})$. Also define for any $x \in X$,
\[
a(x) = \sum_{i=0}^{m-1} \rho(\psi^i x, \psi^{i+1} x).
\]
We will prove that 
\begin{equation}
\rho(\psi^{nm+l} x, \psi^{nm+l+1} x) \leq 2^{n-1} \alpha^n a(x) \tag{3.2}\label{3.2}
\end{equation}
 holds for all $n \in \mathbb{N}$ and $l \in \{0, \dots, m-1\}$. 

To apply the principle of mathematical induction, we need to verify that this inequality holds for $n=1$.\\
For $l \in \{0, \dots, m-2\}$, we have
\begin{align*}
d(\psi^{m+1} x, \psi^{m+l+1} x) &\leq \alpha \left( \rho(\psi^l x, \psi^{l+2} x) + \rho(\psi^{l+1} x, \psi^{l+1} x)\right) \\
&\leq \alpha \left( \rho(\psi^l x, \psi^{l+1} x) + \rho(\psi^{l+1} x, \psi^{l+2} x)\right) \\
&= \alpha (\rho_l + \rho_{l+1}) \\
&\leq \alpha a(x).
\end{align*}
For $l = m-1$, we have
\begin{align*}
\rho(\psi^{2m-1} x, \psi^{2m} x) &\leq \alpha \left(\rho(\psi^{m-1} x, \psi^{m+1} x) + \rho(\psi^m x, \psi^m x)\right) \\
&\leq \alpha \left(\rho(\psi^{m-1} x, \psi^m x) + \rho(\psi^m x, \psi^{m+1} x)\right) \\
&= \alpha (\rho_{m-1} + \rho_m).
\end{align*}
But $\rho_m$ itself can be bounded using the previous case with $l=0$: $\rho_m \leq \alpha(\rho_0 + \rho_1)$. Hence
\[ \rho(\psi^{2m-1} x, \psi^{2m}x) \leq \alpha \rho_{m-1}+ \alpha^2 (\rho_0 + \rho_1)\leq \alpha a(x). \]
$$\implies \rho(\psi^{m+l} x, \psi^{m+l+1} x) \leq \alpha a(x).$$\\
Assume that for some $n \geq 1$ and for all $l \in \{0, \dots, m-1\}$,
\[ \rho(\psi^{nm+l} x, \psi^{nm+l+1} x) \leq 2^{n-1} \alpha^n a(x). \]

In order to prove that (\ref{3.2}) holds for any $n \in \mathbb{N}$ and $l \in \{0, \dots, m-1\}$, observe that
\begin{align*}
\rho(\psi^{(n+1)m+l} x, \psi^{(n+1)m+l+1} x) &\leq \alpha \left(\rho(\psi^{nm+l} x, \psi^{nm+l+2} x)+\rho(\psi^{nm+l+1} x, \psi^{nm+l+1} x)\right) \\
&\leq \alpha \left( \rho(\psi^{nm+l} x, \psi^{nm+l+1}x) + \rho(\psi^{nm+l+1} x, \psi^{nm+l+2}x )\right) \\
&\leq \alpha ( 2^{n-1} \alpha^n a(x) + 2^{n-1} \alpha^n a(x) ) \\
&= 2^n \alpha^{n+1} a(x).
\end{align*}
holds for any $l \in \{0, 1, \dots, m-2\}$.\\
If $l = m-1$, then
\begin{align*}
\rho(\psi^{(n+2)m-1} x, \psi^{(n+2)m} x) &\leq \alpha\left( \rho(\psi^{(n+1)m-1} x, \psi^{(n+1)m+1} x) \right)\\
&\leq \alpha \left(\rho(\psi^{(n+1)m-1} x, \psi^{(n+1)m} x) + \rho(\psi^{(n+1)m} x, \psi^{(n+1)m+1} x)\right) \\
&\leq \alpha \rho(\psi^{(n+1)m-1} x, \psi^{(n+1)m} x) + \alpha \rho(\psi^{nm} x, \psi^{nm+2} x) \\
&\leq \alpha \left( 2^{n-1} \alpha^n a(x) + 2^{n} \alpha^{n+1} a(x) \right) \\
&= 2^{n-1} \alpha^{n+1} (1 + 2\alpha) a(x) \\
&< 2^{n-1} \alpha^{n+1} a(x).
\end{align*}
Thus, by the principle of mathematical induction, we find that (\ref{3.2}) holds for any $n \in \mathbb{N}$ and $l \in \{0, \dots, m-1\}$.

For any $n_2, n_1 \in \mathbb{N}$ with $n_2 \geq n_1$, $k_i=\left[\frac{n_i}{m}\right]$ write $n_i = k_im + l_i$ with $l_i \in \{0, \dots, m-1\}$. Then
\begin{align*}
\rho(x_{n1}, x_{n_2}) &\le \sum_{i=n_1}^{n_2-1} \rho_i \\
&\le \sum_{i=k_1m}^{k_2m+m-1} \rho_i \\
&\le 2^{k_1-1} \alpha^{k_1} (m - l_1) a(x_0) + \sum_{i=k_1+1}^{k_2-1} 2^{i-1} \alpha^i m a(x_0) + 2^{k_2-1} \alpha (l_2 + 1) a(x_0) \\
&\le m a(x_0) \sum_{i=k_1}^{\infty} 2^{i-1} \alpha^i.
\end{align*}
Since $2\alpha < 1$, the series $\sum_{i=1}^{\infty} 2^{i-1} \alpha^i$ converges. So, $\lim \limits_{m,n \to \infty} \rho(x_m, x_n) = 0$. By completeness of $X$, there exists $x^* \in X$ such that $\lim \limits_{n \to \infty} x_n = x^*$.\\
Moreover, if $y_0 \in X$ be arbitrary, then in a similar manner we will obtain that $(\psi^ny)$ is a convergent sequence and 
\[ \rho(x_{nm}, y_{nm}) \le \alpha \left( \rho(x_{(n-1)m}, x_{(n-1)m+1}) + \rho(y_{(n-1)m}, y_{(n-1)m+1})\right), \]
and by iteration $\rho(x_{nm}, y_{nm}) \to 0$ as $n \to \infty$. Thus all iterative sequences converge to the same limit $x^*$ for any $y \in X$.\\
Since $\psi$ is $k$-continuous, $\psi^k$ is continuous. So
\[ \psi^k x^* = \psi^k \left(\lim \limits_{n \to \infty} \psi^nx^*\right) = \lim \limits_{n \to \infty} \psi^{k+n}x^* = x^*. \]
Note that for any $m \in \mathbb{N}$,
\[\rho(x^*, \psi x^*)=\rho(\psi^{kn}x^*, \psi^{kn+1} x^*)\leq \sum_{i=kn}^{kn+1}\rho(\psi^ix^*, \psi^{i+1} x^*)\to 0~\text{as}~n\to\infty.\]
Since the series $\sum \rho_i$ converges. Therefore $\psi x^* = x^*$.\\
Suppose $y \in X$ satisfy $\psi y = y$. Then
\[ \rho(x^*, y) = \rho(\psi^m x^*, \psi^m y) \le \alpha \left(\rho(x^*, \psi y) + \rho(y, \psi x^*)\right) = 2\alpha \rho(x^*, y). \]
Since $2\alpha < 1$, we must have $\rho(x^*, y) = 0$ and hence $x^*$ is a unique fixed point of the mapping $\psi$.
\end{proof}
\begin{corollary}
Let $(X, \rho)$ be a complete metric space and $\psi : X \to X$ is a mapping such that for any $x, y \in X$ satisfying (\ref{3.1}) for some $\alpha \in \left[0, \frac{1}{2}\right)$ and $m \ge 3$. Then for any arbitrary point $x \in X$, the iterative sequence $(\psi^n x)$ converges in $(X, \rho)$ to the same point $x^* \in X$ and $Fix(\psi) \subseteq \{x^*\}$.
\end{corollary}
\begin{corollary}
An $m$-Chatterjea contraction for $m \ge 3$ on a complete metric space has atmost one fixed point. 
\end{corollary}

We now give an example of $3$-Chatterjea contraction which is neither a Chatterjea nor a $2$-Chatterjea contraction and has no fixed point, illustrating the necessity of the continuity hypothesis.

\begin{example}\label{C}
 Let $X = [0, 1]$ with the Euclidean metric and define $\psi : X \to X$ by
\[
\psi x = 
\begin{cases} 
\frac{1}{8}, & x = 0, \\
\frac{x}{8}, & x \in (0, 1]. 
\end{cases}
\]
The mapping $\psi$ is neither a Chatterjea contraction nor a $2$-Chatterjea contraction.
However
\[
\psi^3 x = 
\begin{cases} 
\frac{1}{512}, & x = 0, \\
\frac{x}{512}, & x \in (0, 1]
\end{cases}
\]
satisfy (\ref{3.1}). In order to prove that we will treat three cases:\\
\begin{enumerate}
\item If $x, y \in (0, 1]$, then
\begin{align*}
    \rho(\psi^3 x, \psi^3 y) &= \left| \frac{x-y}{512} \right| \\
    &= \frac{1}{8} \frac{|x-y|}{64} \\
    &\le \frac{1}{8} \frac{x+y}{64} \\
    &\le \frac{1}{8} (\rho(x, \psi y) + \rho(y, \psi x)).
\end{align*}
\item If $x=0$, $y \in (0,1]$ (Symmetric case is analogous), then
\begin{align*}
    \rho(\psi^3 x, \psi^3 y) &= \frac{1-y}{512} \\
    &= \frac{1}{8} \left( \frac{1}{64} - \frac{y}{64} \right) \\
    &\le \frac{1}{8} \rho(x, \psi y) \\
    &\le \frac{1}{8} (\rho(x, \psi y) + \rho(y, \psi x)).
\end{align*}
\end{enumerate}
Thus, $\psi$ is a $3$-Chatterjea contraction but none of the iterates of $\psi$ possess a fixed point in $X$.
\end{example}
\begin{corollary}
If $3$-Chatterjea contraction on a complete metric space does not have a fixed point, then the unique limit point of all iterative sequences $(\psi^n x)$ for any $x \in X$ is the point of discontinuity of the mapping $\psi$.
\end{corollary}
\begin{example}
Analyzing the Example \ref{C}, for $x > 0$, $\psi x = \frac{x}{8} = x \Rightarrow x = 0$, not possible. For $x = 0, \psi 0 = \frac{1}{8} \ne 0$. Thus no fixed point exists.\\
 For any $x \in (0,1], \psi^n x = \frac{x}{8^n} \to 0$ as $n \to \infty$ and from $x = 0,$ we get $\psi^n 0 = \frac{1}{8^n} \to 0$. Thus every iterative sequence converges to $0$.
Observe that $\psi$ is discontinuous at $0$. Take $x_n = \frac{1}{8^n} \to 0$. Then $\psi_{x_n} = \frac{1}{8^{n}} \to 0$ but $\psi_0 = \frac{1}{8} \neq 0$. Moreover,
for any $k \in \mathbb{N}$, $\psi^k$ is also discontinuous at $0$ because $\psi^k(\frac{1}{n}) = \frac{1}{8^kn} \to 0$ while $\psi^k_0 = \frac{1}{8^k} \neq 0$.
This demonstrate the necessity of $k$-continuity hypothesis in Theorem \ref{D}.
\end{example}
From the proof of Theorem \ref{D}, we can find that it is sufficient for $\psi^3$ to be continuous at the limit of iterative sequence
$(\psi^n)$ in order to have fixed point. In general, it is not a necessary assumption.

\begin{example}\label{E}
Let $X = [0, 1]$ be equipped with Euclidean metric. Define mapping $\psi : X \to X$ as follows:
$$
\psi_x = 
\begin{cases} 
0, & x = 0, \\
0, & x = \frac{1}{4}, \\
\frac{1}{4}, & x = \frac{1}{2}, \\
\frac{1}{2}, & x \in [0, 1] \setminus \{0, \frac{1}{4}, \frac{1}{2}\}.
\end{cases}
$$
For any sequence $x_n$ with $x_n \to \frac{1}{2}$ and
$x_n \neq \frac{1}{2}$ $\left(e.g., x_n = \frac{1}{2} + \frac{1}{n}, n \geq 3\right)$, we have
$\psi(x_n) = \frac{1}{2}$ for all $n$. Hence $\lim \limits_{n \to \infty} \psi(x_n) = \frac{1}{2}$.
However, $\psi(\frac{1}{2}) = \frac{1}{4}$. Thus $\psi$ is not continuous
at $x = \frac{1}{2}$.
We now compute $\psi^3$ for every $x \in [0, 1]$:
\begin{enumerate}
\item If $x = 0$, then\\
$\psi(0) = 0$, $\psi^2(0) = \psi(0) = 0$, and $\psi^3(0) = \psi(0) = 0$.
\item If $x = \frac{1}{4}$, then\\
$\psi(\frac{1}{4}) = 0$, $\psi^2(\frac{1}{4}) = \psi(0) = 0$ and $\psi^3(\frac{1}{4}) = \psi(0) = 0$.
\item If $x = \frac{1}{2}$, then\\
$\psi(\frac{1}{2}) = \frac{1}{4}$, $\psi^2(\frac{1}{2}) = \psi(\frac{1}{4}) = 0$, and $\psi^3(\frac{1}{2}) = \psi(0) = 0$.

\item For any other $x \in [0, 1] \setminus \{0, \frac{1}{4}, \frac{1}{2}\}$, we have\\
$\psi(x) = \frac{1}{2}$, then $\psi^2(x) = \psi(\frac{1}{2}) = \frac{1}{4}$ and $\psi^3(x) = \psi(\frac{1}{4}) = 0$.
\end{enumerate}
The constant function $x \mapsto 0$ is continuous on $[0, 1]$.
Thus $\psi$ is $3$-continuous and also $\psi$ is a $3$-Chatterjea contraction with $\alpha = \frac{1}{3}$ and has 0 as a unique fixed point.
\end{example}
\begin{remark}
The Example \ref{E} demonstrates that the $k$-continuity hypothesis in Theorem \ref{D} does not force $\psi$ itself to be continuous, it only requires some iterate of $\psi$ to be continuous. 
\end{remark}
\section*{Conflict of Interest}
The authors declare that there is no conflict of interest regarding the publication of this paper.

\end{document}